\documentclass[12pt]{amsart}

\usepackage{amsmath,amssymb,amsthm,enumerate,color}

\usepackage[T1]{fontenc}
\usepackage[utf8]{inputenc}
\usepackage{lmodern}
\usepackage[margin=1.04in]{geometry}
\usepackage{amsmath,amssymb,amsthm,mathtools}
\usepackage{enumitem}
\usepackage{xcolor}
\usepackage[hidelinks]{hyperref}

\usepackage{needspace}

\allowdisplaybreaks

\definecolor{darkblue}{RGB}{22,55,105}
\hypersetup{
  colorlinks=true,
  linkcolor=darkblue,
  citecolor=darkblue,
  urlcolor=darkblue
}

\usepackage[capitalise]{cleveref}
\crefname{theorem}{Theorem}{Theorems}
\crefname{lemma}{Lemma}{Lemmas}
\crefname{corollary}{Corollary}{Corollaries}
\crefname{proposition}{Proposition}{Propositions}
\crefname{conjecture}{Conjecture}{Conjectures}
\crefname{question}{Question}{Questions}
\crefname{definition}{Definition}{Definitions}
\crefname{example}{Example}{Examples}
\crefname{remark}{Remark}{Remarks}
\crefname{enumi}{}{}
\crefname{equation}{}{}

\newtheorem{theorem}{Theorem}[section]
\newtheorem*{theorem*}{Theorem}

\newtheorem*{conjecture*}{Conjecture}
\newtheorem*{tconjecture*}{The Trautman Conjecture}
\newtheorem{corollary}[theorem]{Corollary}
\newtheorem*{corollary*}{Corollary}
\newtheorem{lemma}[theorem]{Lemma}
\newtheorem*{lemma*}{Lemma}
\newtheorem{proposition}[theorem]{Proposition}
\newtheorem*{proposition*}{Proposition}

\theoremstyle{definition}

\theoremstyle{remark}
\newtheorem{remark}[theorem]{Remark}

\newcommand{\beq}{\begin{equation}}
\newcommand{\eeq}{\end{equation}}

\def\sideremark#1{\ifvmode\leavevmode\fi\vadjust{\vbox to0pt{\vss
 \hbox to 0pt{\hskip\hsize\hskip1em
 \vbox{\hsize2.7cm\tiny\raggedright\pretolerance10000
  \noindent #1\hfill}\hss}\vbox to8pt{\vfil}\vss}}}

\newcommand{\C}{\mathbb C}
\newcommand{\R}{\mathbb R}
\newcommand{\ii}{\mathrm i}

\newcommand{\SpanC}{\operatorname{span}_{\C}}
\DeclareMathOperator{\diver}{div}
\DeclareMathOperator{\supp}{supp}

\begin{document}
\title[]{A Smooth Counterexample to\\ the Trautman Conjecture}
\author{Sean N. Curry}
\address{Department of Mathematics, Oklahoma State University, Stillwater, OK 74078-5061}
\email{sean.curry@okstate.edu}


\begin{abstract}
The Trautman conjecture asserted that a smooth three-dimensional CR manifold admitting a nowhere-zero closed section of its canonical bundle must be locally embeddable in $\C^2$.  We modify a standard construction of nonembeddable smooth strongly pseudoconvex CR 3-manifolds so that the condition of having a nowhere-zero closed section of the canonical bundle is preserved, thus providing a strongly pseudoconvex counterexample to the Trautman conjecture.  
\end{abstract}

\maketitle

\section{Introduction and main results}\label{sec:introduction}

Let $M$ be a smooth $3$-dimensional CR manifold with CR line bundle
$T^{0,1}\subset\C TM$.  Its \emph{canonical bundle} is the complex line bundle
\begin{equation*}
 K=\{\,\Omega\in\C\Lambda^2T^*M:
       \iota_L\Omega=0\text{ for every }L\in T^{0,1}\,\}.
\end{equation*}

This paper constructs a counterexample to the following well-known conjecture:

\begin{tconjecture*}[\cite{Trautman1999}]
If a smooth CR
$3$-manifold admits a nowhere-zero closed section of the canonical bundle, then it is
locally CR embeddable in $\C^2$.
\end{tconjecture*}

\subsection{Origins of the Conjecture}
The Trautman conjecture arose naturally out of the work of Robinson, Trautman and many others on exact solutions of the Maxwell and Einstein equations via dimensional reduction, going back to the late 1950s and early 1960s \cite{Robinson1959,Robinson1961,RobinsonTrautman1960,
RobinsonTrautman1962}. The link to Cauchy--Riemann (CR) geometry was provided by the key observation that a $4$-dimensional Lorentzian manifold admitting a shear-free congruence of null geodesics induces a CR structure on the local $3$-dimensional leaf space of this congruence \cite{Nurowski1996,RobinsonTrautman1986,Trautman2002}. 
This CR structure is strongly pseudoconvex when the congruence is \emph{twisting} (which is, locally, a generic condition). Under suitable hypotheses, the Maxwell and Einstein equations can then be reinterpreted in terms of the Cauchy--Riemann operator on the underlying CR $3$-manifold \cite{Trautman1999,HillLewandowskiNurowski2008,LewandowskiNurowskiTafel1990,LewandowskiNurowskiTafel1991,SchmalzGanji2019}. 
When this CR $3$-manifold is taken to be embedded in $\mathbb{C}^2$, the prevalence of CR functions (and closed sections of the canonical bundle) coming from ambient holomorphic functions allows for the construction of many explicit solutions. This framework is natural in the sense that every algebraically special vacuum Einstein spacetime admits such a geodesic congruence (the Goldberg-Sachs theorem \cite{GoldbergSachs1962,GoldbergSachs2009}), as does every spacetime admitting a null electromagnetic field (the Robinson theorem \cite{Robinson1961,Tafel1985}). Celebrated examples of metrics which arise from shear-free null geodesic congruences in this way include the Kerr and Kerr--Newman black holes. 

The Trautman conjecture concerns a specific construction of $4$-dimensional Lorentzian manifolds admitting null Maxwell fields \cite{RobinsonTrautman1986,Trautman1999,TaghaviChabert2026}, which starts from a CR $3$-manifold together with a nowhere-zero closed section of its canonical bundle. Trautman conjectured that all such Lorentzian $4$-manifolds locally arise from the data of an embedded hypersurface in $\mathbb{C}^2$ (and a closed section of the canonical bundle obtained via the embedding). From this point of view, the counterexample construction below therefore shows that the scope for using $3$-dimensional CR manifolds to construct $4$-dimensional Lorentzian manifolds admitting null electromagnetic fields is larger than was expected.

When formulating his conjecture, Trautman was influenced by the then-current state of knowledge concerning embeddability of CR structures \cite{Trautman1999}. In particular, he noted that Rosay had constructed nonembeddable CR $3$-manifolds for which all local CR functions depend on a single CR function \cite{Rosay1989} and that, by the results of Tafel and Jacobowitz, these examples cannot carry a nowhere-zero closed section of the canonical bundle \cite{Tafel1985,Jacobowitz1987,Trautman1999}; he credited C.~Denson Hill with pointing out the relevance of Rosay's examples and Jacobowitz's results. The known $3$-dimensional examples therefore did not contradict the proposed implication. On the other hand, Jacobowitz had constructed nonembeddable CR manifolds of real dimension seven whose canonical bundles do admit nowhere-zero closed sections \cite{Jacobowitz1987}. Trautman nevertheless conjectured that the existence of a nowhere-zero closed section of the canonical bundle would imply local embeddability in the exceptional $3$-dimensional case.

\subsection{Main Results}
We now state our main result.  Let $(z,t)$ be the coordinates on  $\C\times\R$ and define
\begin{equation*}
 L_0=\partial_{\bar z}-\frac{\ii}{2}z\partial_t,
 \qquad
 w=t+\frac{\ii}{2}|z|^2. 
\end{equation*}
%
\begin{theorem}\label{thm:main}
For every neighborhood $U$ of the origin in $\C\times\R$ and every
$\varepsilon>0$, there is a nonnegative function
$\phi\in C^\infty(\C\times\R)$ supported in $U$, with $\|\phi\|_{C^2}<\varepsilon$, such that the complex line field
\begin{equation}\label{eq:Lphi-intro}
 T^{0,1}_\phi=\SpanC\{L_\phi\},
 \qquad
 L_\phi=L_0+\phi_t\partial_z-\phi_z\partial_t,
\end{equation}
has the following properties:
\begin{enumerate}[label=\textnormal{(\roman*)},leftmargin=2.4em]
\item $T^{0,1}_\phi$ is a smooth strongly pseudoconvex CR structure;
\item its canonical bundle has the nowhere-zero closed section
\begin{equation*}
 \Omega_\phi
 =\iota_{L_\phi}(dz\wedge d\bar z\wedge dt)
 =-dz\wedge dw-d(\phi\,d\bar z);
\end{equation*}
\item every $C^1$ CR function $h$ defined near the origin satisfies
$dh(0)=0$.
\end{enumerate}
In particular, $T^{0,1}_\phi$ is not locally CR embeddable at the origin.
\end{theorem} 
\begin{remark}
In \cite{Jacobowitz1987} (cf.\ \cite{Jacobowitz2020}) Jacobowitz has shown that a strongly pseudoconvex CR 3-manifold which admits a nowhere-zero closed section of the canonical bundle is locally embeddable in a neighborhood of a given point if and only if it admits a single strongly independent CR function in a neighborhood of that point. Note that while \cref{thm:main}(iii) directly implies nonembeddability at the origin, it also implies the absence of a strongly independent CR function near $0$.
\end{remark}

In the family of examples that we construct, $\phi$ vanishes to infinite order at the origin. Hence $T^{0,1}_\phi$ agrees with the Heisenberg CR structure outside $U$ and has the same infinite jet as the Heisenberg structure at the origin. Our construction is very similar to those of Nirenberg \cite{Nirenberg1974,Nirenberg1975}, Jacobowitz--Tr\`eves \cite{JacobowitzTreves1982} and Hill-Nacinovich \cite{HillNacinovich2013} but differs in that the CR structure \cref{eq:Lphi-intro} depends on the derivatives of a single function $\phi$ (rather than simply on $\phi$ itself). The Hamiltonian-like structure of the perturbation term $\phi_t\partial_z-\phi_z\partial_t$ was chosen to ensure that the resulting CR structure admits a nowhere-zero closed section of the canonical bundle. The proof of nonembeddability is similar to that in \cite{JacobowitzTreves1982,Nirenberg1974,Nirenberg1975}, except that we are forced to introduce weight functions $z$ and $w$ and integrate by parts in order to remove the derivatives from the smooth bump function $\phi$. Note that smoothness is essential here; it is well-known that a real analytic strongly pseudoconvex CR structure is always locally embeddable.

Since the existence of a nonvanishing closed section of the canonical bundle is locally equivalent to the existence of a pseudo-Einstein contact form, \cref{thm:main} has the following immediate corollary:

\begin{corollary}\label{cor:pseudo-einstein}
There is a smooth strongly pseudoconvex CR 3-manifold which admits a
pseudo-Einstein contact form but is not locally embeddable in $\C^2$.
\end{corollary}

\begin{remark}
In dimension three the standard definition of a pseudo-Einstein contact form comes from \cite{CaseYang2013,Hirachi2013}. This notion was first introduced in higher dimensions by Lee \cite{Lee1988}. The existence of a pseudo-Einstein contact form is locally equivalent to the existence of a nowhere-zero closed section of the canonical bundle in all dimensions. 
\end{remark}

For embedded CR manifolds, the pseudo-Einstein condition is also related
to Fefferman's complex Monge--Amp\`ere equation, whose exact solution on
a bounded strongly pseudoconvex domain gives the Cheng--Yau complete
K\"ahler--Einstein metric \cite{Fefferman1976,ChengYau1980,Farris1986,Hirachi2013}. The pseudo-Einstein condition also simplifies the formal construction of a K\"ahler--Einstein structure with prescribed CR infinity in \cite{BiquardHerzlich2005}. By results of \cite{BiquardHerzlich2005,Cho1998,Kiremidjian1979}, however, a smooth local K\"ahler--Einstein extension on the pseudoconvex side would yield a local two-sided extension and hence local CR embeddability. \cref{thm:main} therefore shows that the pseudo-Einstein condition alone does not guarantee such an extension.

A complex surface is called \emph{unimodular} if it is endowed with a holomorphic volume form $\Omega$. A holomorphic volume form (such as $dz\wedge dw$ on $\mathbb{C}^2$) pulls back to a nowhere-zero closed section of the canonical bundle on any embedded strongly pseudoconvex CR hypersurface. Consequently, Bryant \cite{Bryant2004} terms a strongly pseudoconvex CR $3$-manifold endowed with a nowhere-zero closed section of the canonical bundle a \emph{unimodular} CR structure. In \cite{Bryant2004} Bryant considered the natural embedding problem for unimodular CR structures and showed that if such a structure is real analytic then it is always locally realizable as a hypersurface in a unimodular complex surface. \cref{thm:main} implies that this result cannot be extended to smooth unimodular CR structures:

\begin{corollary}\label{cor:unimodular}
There is a smooth unimodular CR structure that is not locally realizable as a hypersurface in a unimodular complex surface.
\end{corollary}
\begin{remark}
In \cite{Bryant2004} Bryant gives a geometric flow equation on a unimodular CR $3$-manifold $M$ whose solvability for $t\in (-\epsilon,\epsilon)$ implies that $M\times (-\epsilon,\epsilon)$ admits a unimodular complex surface structure inducing the given unimodular CR structure on $M\times \{0\}$. While for the counterexamples constructed in this paper the Bryant flow cannot have a smooth two-sided solution, it remains possible that a one-sided solution always exists realizing the unimodular CR structure as the pseudoconcave boundary of a unimodular complex surface (i.e.\ locally, there may always be a solution for $t<0$ small).
\end{remark}

In a more positive direction, the counterexample constructed shows that the class of strongly pseudoconvex CR $3$-manifolds possessing a nowhere-zero closed section of the canonical bundle required for the various $4$-dimensional spacetime constructions given in, e.g., \cite{AlekseevskyGanjiSchmalz2018,HillLewandowskiNurowski2008,Trautman1999,Trautman2002} is considerably larger than was previously expected.

\subsection{Globalization}
We conclude the introduction by noting that the examples constructed in \cref{thm:main} are standard at infinity and therefore extend, via the Cayley compactification, to $S^3$ as small nonembeddable perturbations of the standard spherical CR structure. Clearly such CR structures admit a nowhere-zero closed section of the canonical bundle on $S^3\setminus\{\infty\}$. A slight modification of this compactification gives the following global result.
\begin{theorem}\label{thm:globalization}
Given a point $p\in S^3$, there exists a smooth perturbation of the standard spherical CR structure on $S^3$ that admits a global nowhere-zero closed section of the canonical bundle but is not embeddable in any neighborhood of $p$. The perturbation can be taken to be arbitrarily small in the $C^1$ norm.
\end{theorem}

In particular, it follows that there is a pseudo-Einstein CR structure near the standard structure on $S^3$ that is not locally embeddable. Moreover, from the proof in \cref{sec:globalization} below and the characterization of pseudo-Einstein contact forms in terms of volume normalization \cite{CaseYang2013,Hirachi2013,Lee1988} it follows that the pseudo-Einstein contact form can be taken to be a small perturbation of the standard contact form. 


\subsection*{Acknowledgments} The counterexample construction presented below was discovered through experimentation using ChatGPT Plus on August 18, 2026. ChatGPT was also used to create an initial rough draft of this manuscript (which has since been thoroughly revised) and in the proofreading process. The author independently checked all calculations and arguments presented here and takes full responsibility for the contents of the paper.

The author would like to thank Pawe\l\ Nurowski and Rod Gover for directing him to this problem when he was a student. He would also like to thank Howard Jacobowitz, Gerd Schmalz and Arman Taghavi-Chabert for helpful conversations and for sharing their skepticism regarding the Trautman conjecture.

\section{Hamiltonian perturbations of the Heisenberg CR structure}\label{sec:hamiltonian}

In the presence of a fixed volume form, the Trautman conjecture can be rephrased in terms of a vector field dual to the closed section of the canonical bundle \cite{Trautman1999}. To see this we write $z=x+iy$ on $\mathbb{C}\times \mathbb{R}$ and fix the background volume form $dV=dx\wedge dy\wedge dt$. For convenience, we will make use of both $dV$ and the complex volume form
\begin{equation*}
 \nu=dz\wedge d\bar z\wedge dt = -2\ii\,dV.
\end{equation*}
Since $\nu$ is a constant multiple of $dV$ it defines the same notion of divergence as $dV$ (namely the standard divergence operator on $\mathbb{R}^3$). Hence for any real or complex vector field $X$
$$
\mathcal{L}_X \nu = (\mathrm{div}\,X)\,\nu.
$$
By Cartan's formula, it follows that $\mathrm{div}\, X =0$ if and only if $\iota_X\nu$ is closed. Thus, when $X$ is a $(0,1)$-vector field, $\mathrm{div}\, X =0$ if and only if $\Omega = \iota_X\nu$ is a closed section of the canonical bundle.

 For a smooth real function $\phi$, set
\begin{equation}
 X_\phi=\phi_t\partial_z-\phi_z\partial_t,
 \qquad L_\phi=L_0+X_\phi.
\end{equation}

\begin{proposition}\label{prop:closed-section}
The field $L_\phi$ is divergence-free relative to $\nu$. If $\phi$ is supported in a fixed precompact subset of $\C\times \R$ and $\|\phi\|_{C^2}$ is sufficiently small, then $T^{0,1}_\phi$ is a strongly pseudoconvex CR structure on $\C\times \R$ and
\begin{equation*}
 \Omega_\phi:=\iota_{L_\phi}\nu
 =-dz\wedge dw-d(\phi\,d\bar z)
\end{equation*}
is then a closed section of the canonical bundle of
$T^{0,1}_\phi=\SpanC\{L_\phi\}$.  Moreover, the form $\Omega_\phi$ is nowhere zero.  
\end{proposition}

\begin{proof}
Direct computation gives
\begin{equation*}
 \diver X_\phi=\partial_z\phi_t-\partial_t\phi_z=0,
\end{equation*}
and that $L_0$ is also divergence-free. It follows that $\Omega_{\phi}$ is closed. 

When $\phi=0$ one has
\begin{equation*}
 [L_0,\overline{L_0}]=\ii\partial_t.
\end{equation*}
Linear independence of $L_0$ and $\overline{L_{0}}$ and nondegeneracy of the Levi form are open conditions in the $C^1$ topology on the defining field, hence in the $C^2$ topology on $\phi$.  It follows that $T^{0,1}_\phi$ is a strongly pseudoconvex CR structure when $\|\phi\|_{C^2}$ is small.

Contracting with $\nu$ we have
\begin{align*}
 \iota_{L_0}\nu
 &=-dz\wedge dt-\frac{\ii}{2}z\,dz\wedge d\bar z
   =-dz\wedge dw,\\
 \iota_{X_\phi}\nu
 &=\phi_t\,d\bar z\wedge dt-\phi_z\,dz\wedge d\bar z
   =-d(\phi\,d\bar z),
\end{align*}
and hence $\Omega_{\phi}=\iota_{L_\phi}\nu
 =-dz\wedge dw-d(\phi\,d\bar z)$.  Clearly, 
$\iota_{L_\phi}\Omega_\phi=0$ so $\Omega_\phi$ is a section of the canonical bundle. The coefficient of $dz\wedge dt$ in $\Omega_{\phi}$ is $-1$, so $\Omega_\phi$ never vanishes.    
\end{proof}

\section{Bump functions on shrinking solid tori}
\label{sec:tori}

Here we construct the desired scalar perturbation function $\phi$. For $j\geq 0$ let
\begin{equation*}
 a_j=2^{-j},
\qquad  
A_j=\left\{w:|w-\ii a_j|< \vphantom{\tfrac12} ca_j\right\},
\qquad
H_j=\left\{w:|w-\ii a_j|< \tfrac12 ca_j\right\},
\end{equation*}
where $0<c<\frac{1}{3}$. Note that every disk $A_j$ lies in the upper half-plane and the disks accumulate only at the
origin.

Consider the map $w:\C\times\R\to\C$ given by $w=t+\frac{\ii}{2}|z|^2$, as above. This map has circular fibers over points in the upper half-plane that shrink to points as one approaches the real axis.  Abusing notation by writing $w$ also for a point in the image plane, the fiber over a point $w=u+\ii v$ with $v>0$ is then
\begin{equation}\label{eq:fiber}
 \Gamma_w=\{t=u,\ |z|^2=2v\}\simeq S^1.
\end{equation}
Thus
\begin{equation*}
 D_j=w^{-1}(A_j)
\end{equation*}
is the interior of a solid torus.  The sets $D_j$ are pairwise disjoint and
shrink to the origin.

Choose a nonnegative, nonzero function $\rho\in C_c^\infty(\C)$ with
\begin{equation*}
 \supp\rho\subset\{\zeta:|\zeta|<1/2\}.
\end{equation*}
Define the explicit bump functions
\begin{equation}\label{eq:bumps}
 \phi_j(z,t)=e^{-1/a_j^2}
 \rho\!\left(\frac{w(z,t)-\ii a_j}{ca_j}\right),
 \qquad
 \phi=\delta\sum_{j\ge j_0}\phi_j,
\end{equation}
where $j_0$ will be taken large and $\delta>0$ small.  Then
\begin{equation*}
 \supp\phi_j\subset w^{-1}(H_j)\Subset D_j.
\end{equation*}
%

\begin{lemma}\label{lem:smallness}
The function $\phi$ in \cref{eq:bumps} is smooth and vanishes to infinite order at the origin. 
Given a prescribed neighborhood $U$ of the origin and
$\varepsilon>0$, the parameters $j_0$ and $\delta$ can be chosen so that
$\supp\phi\subset U$ and $\|\phi\|_{C^2}<\varepsilon$. 
\end{lemma}

\begin{proof}
On $D_j$ one has $|t|=O(a_j)$ and $|z|=O(a_j^{1/2})$.  For every multi-index
$\alpha$, differentiation of the rescaled bump gives the bound
\begin{equation}\label{eq:derivative-bound}
 \|\partial^\alpha\phi_j\|_{L^\infty}
 \le C_\alpha e^{-1/a_j^2}a_j^{-|\alpha|}.
\end{equation}
The exponential factor dominates every inverse power of $a_j$.  Since the
distance of $D_j$ from the origin is comparable to $a_j^{1/2}$, every derivative
of $\phi_j$ tends to zero faster than every power of that distance.  The
supports are disjoint, so the sum is locally finite away from the origin and
extends across the origin with zero infinite jet.  Taking $j_0$ large puts all
supports in $U$, and the final factor $\delta$ gives the asserted $C^2$
smallness.
\end{proof}

\section{Weighted moment identities}\label{sec:period}

Here we apply ideas from \cite{Lewy1957,Nirenberg1975,JacobowitzTreves1982,HillNacinovich2013} to establish a key family of weighted moment identities that will be used in the next section to prove that $L_{\phi}$ is not locally embeddable.

Given a point $w = u+iv$ in the upper half-plane and $\chi(z,\bar{z},t)$ a $C^1$ function on $\C\times\R$, we define the fiber period to be
\begin{equation}\label{eq:period-definition}
 I_\chi(w)=\oint_{\Gamma_w}\chi\,dz,
\end{equation}
where $\Gamma_{w}$ is as in \cref{eq:fiber}. Setting $r=(2v)^{1/2}$ we therefore have
\begin{equation}\label{eq:period-formula}
 I_\chi(w)=\int_0^{2\pi}\ii re^{\ii\vartheta}\chi(re^{\ii\vartheta},re^{-\ii\vartheta},u)\,d\vartheta.
\end{equation}

\begin{lemma}\label{lem:fiber-period}
Let $G$ be a connected open subset of the upper half-plane that contains a one-sided neighborhood $\{\,u+iv \,\colon\, u\in J,\, v\in (0,\eta) \, \}$ of an open interval $J$ in the real line for some $\eta>0$.  Suppose that $\chi$ is $C^1$
in a neighborhood of $\overline{w^{-1}(G)}$ and satisfies
\begin{equation}\label{eq:L0chi-zero}
 L_0\chi=0\qquad\text{on }w^{-1}(G).
\end{equation}
Then
\begin{equation}\label{eq:period-zero}
 I_\chi(w)=0\qquad\text{for all } w\in G.
\end{equation}
If, in addition, $A$ is an open disk compactly contained in the
upper half-plane, $\partial A\subset G$, and the same $\chi$  is $C^1$
in a neighborhood of $\overline{w^{-1}(A)}$, then
\begin{equation}\label{eq:bulk-zero}
 \int_{w^{-1}(A)}L_0\chi\,dV=0.
\end{equation}
\end{lemma}

\begin{proof}
Differentiating \cref{eq:period-formula} with respect to $u$ we have
\begin{equation} \label{eq:Iu}
\partial_u I_\chi
 =\int_0^{2\pi}\ii z\chi_t\,d\vartheta.
\end{equation}
Similarly, since $r=(2v)^{1/2}$ implies $\partial_v = \frac{1}{r}\partial_r$, we have
\begin{equation} \label{eq:Iv}
 \partial_v I_\chi
 =\frac{\ii}{r}\int_0^{2\pi}
   \partial_r (z\chi)
   \,d\vartheta = \frac{\ii}{r^2}\int_0^{2\pi}
   \bigl(z\chi+z^2\chi_z+r^2\chi_{\bar z}\bigr)
   \,d\vartheta,
\end{equation}
where $z=re^{\ii \vartheta}$. Now, integrating an angular derivative from $0$ to $2\pi$ gives zero, and hence
\begin{equation}\label{eq:angular-derivative}
 0=\int_0^{2\pi}\partial_\vartheta(z\chi)\,d\vartheta
 =\ii\int_0^{2\pi}
   \bigl(z\chi+z^2\chi_z-r^2\chi_{\bar z}\bigr)
   \,d\vartheta.
\end{equation}
Combining \cref{eq:Iv} and \cref{eq:angular-derivative} gives
\begin{equation} \label{eq:Iv2}
 \partial_v I_\chi
 =\ii \int_0^{2\pi} 2\chi_{\bar{z}} \,d\vartheta.
\end{equation}
Combining  \cref{eq:Iu} and \cref{eq:Iv2} we therefore obtain
\begin{equation}\label{eq:period-holomorphic}
 2\partial_{\bar w}I_\chi
 =\int_0^{2\pi}
   \bigl(\ii z\chi_t-2\chi_{\bar z}\bigr)\,d\vartheta
 =-2\int_0^{2\pi}L_0\chi\,d\vartheta=0.
\end{equation}
Thus $I_\chi$ is holomorphic on $G$.  When a fiber collapses to the real axis,
\begin{equation}\label{eq:period-collapse}
 |I_\chi(u+\ii v)|
 \le2\pi(2v)^{1/2}\sup_{\Gamma_{u+\ii v}}|\chi|
 \longrightarrow0.
\end{equation}
In particular, $I_{\chi}$ extends continuously to zero on the real open interval $J$. Therefore, $I_\chi$ extends holomorphically across $J$ by Schwarz reflection. The identity theorem then implies that $I_{\chi}$ is zero near $J$, and the connectedness of $G$ gives \cref{eq:period-zero}.

To prove \cref{eq:bulk-zero}, we first note that by integrating fiber-wise and using \cref{eq:period-zero} we have, up to a sign depending on the choices of orientations,
\begin{equation}\label{eq:fiber-integration}
\int_{\partial w^{-1}(A)}\chi\,dz\wedge dw 
 = \pm\int_{\partial A}
    \left(\oint_{\Gamma_w}\chi\,dz\right)dw
 = \pm\int_{\partial A}I_\chi(w)\,dw
 =0.
\end{equation}
Then, since 
\begin{equation}\label{eq:stokes-algebra}
 d(\chi\,dz\wedge dw)
 =-L_0\chi\,dz\wedge d\bar z\wedge dt
 =-L_0\chi\,\nu,
\end{equation}
it follows by Stokes' theorem that
\begin{equation}
 0= \int_{\partial w^{-1}(A)}\chi\,dz\wedge dw
 = \int_{w^{-1}(A)} d(\chi\,dz\wedge dw) = - \int_{w^{-1}(A)}L_0\chi\,\nu.
\end{equation}
Since $\nu=-2\ii\,dV$, this proves \cref{eq:bulk-zero}.
\end{proof}

An important consequence of \cref{eq:bulk-zero} is the following corollary:

\begin{corollary}\label{cor:weighted-bulk}
Let $\phi$ be given by \cref{eq:bumps}, and let $h$ be a  $C^1$ function defined in a neighborhood $V$ of the  origin in $\mathbb{C}\times \mathbb{R}$ satisfying
$L_\phi h=0$.  Then, for every sufficiently
large $j$,
\begin{equation}\label{eq:weighted-bulk}
 \int_{D_j}zL_0h\,dV=0,
 \qquad
 \int_{D_j}wL_0h\,dV=0.
\end{equation}
\end{corollary}

\begin{proof}
Let $U$ be an open disk centered at the origin in the complex plane small enough that $w^{-1}(\overline{U})$ is contained in $V$. Let $U^+$ be the intersection of $U$ with the upper half-plane, and let $G$ be $U^+$ minus the closure of $\bigcup_j H_j$. Note that for all $j$ sufficiently large, $\overline{A_j}$ is contained in $U^+$. Since the closed disks $\overline{H_j}$ are pairwise separated and accumulate only at the boundary point $0$, their complement $G$ is path connected and contains a one-sided neighborhood of an open interval in the real axis (in the sense of \cref{lem:fiber-period}).
 By \cref{eq:bumps}, $\phi=0$ on $w^{-1}(G)$, and hence $L_0h=0$ there. The two standard CR functions $z$ and $w=t+\frac{\ii}{2}|z|^2$ on the Heisenberg group satisfy
\begin{equation}\label{eq:flat-first-integrals}
 L_0z=L_0w=0.
\end{equation}
It follows that $L_0(zh)=L_0(wh)=0$ on $w^{-1}(G)$, and that 
$L_0(zh) = zL_0h$ and $L_0(wh) = wL_0h$ on $w^{-1}(U^+)$.  When $j$ is sufficiently large, $\partial A_j \subset G$.  Applying 
 \cref{lem:fiber-period} with $\chi=zh$ and with $\chi=wh$ therefore gives \cref{eq:weighted-bulk}.
\end{proof}

\section{Weighted moments and nonembeddability}
\label{sec:moments}

We now use the two weighted identities to show that the differential of every local CR function must vanish at the origin.

\begin{proposition}\label{prop:derivative-killing} 
Let $\phi$ be the function in \cref{eq:bumps}.  If $h$ is a $C^1$ function satisfying
$L_\phi h=0$ near the origin, then
\begin{equation}\label{eq:dh-zero}
 dh(0)=0.
\end{equation}
\end{proposition}

\begin{proof}
For every sufficiently large $j$, the solid torus $D_j$ lies in the domain of
$h$.  Because the supports are disjoint, only $\delta\phi_j$ is nonvanishing on
$D_j$, and the equation $L_\phi h=0$ can be written as
\begin{equation}\label{eq:L0h-X}
 L_0h=-\delta X_{\phi_j}h,
 \qquad \text{where }
 X_{\phi_j}=(\phi_j)_t\partial_z-(\phi_j)_z\partial_t.
\end{equation}
The vector field $X_{\phi_j}$ is divergence-free and compactly supported in
$D_j$.

From the first identity in \cref{eq:weighted-bulk} we have
\begin{equation}
0
 =\int_{D_j}zL_0h\,dV
 =-\delta\int_{D_j}zX_{\phi_j}h\,dV.
\end{equation}
Integrating by parts (using that $\mathrm{div}\,X_{\phi_j}=0$), using that $X_{\phi_j}z=(\phi_j)_t$ and then integrating by parts again, we obtain
\begin{equation} \label{eq:t-moment}
 0
 =\int_{D_j}zX_{\phi_j}h\,dV
=-\int_{D_j}hX_{\phi_j}z\,dV
=-\int_{D_j}h(\phi_j)_t\,dV
 =\int_{D_j}\phi_jh_t\,dV.
\end{equation}
Define the positive mass 
\begin{equation}\label{eq:masses}
 M_j:=\int_{D_j}\phi_j\,dV>0.
\end{equation}  
The probability
measures $M_j^{-1}\phi_jdV$ have supports shrinking to the origin.  Continuity
of $h_t$ therefore gives
\begin{equation}\label{eq:ht-zero}
 h_t(0)
 =\lim_{j\to\infty}\frac1{M_j}
   \int_{D_j}\phi_jh_t\,dV
 =0.
\end{equation}

We now carry out a similar calculation, starting with the second identity in \cref{eq:weighted-bulk}. Note first that
\begin{equation}\label{eq:Xw}
 X_{\phi_j}w
 =\frac{\ii}{2}\bar z(\phi_j)_t-(\phi_j)_z.
\end{equation}
Imitating the calculation in the previous paragraph therefore gives
\begin{align}
 0
 &=\int_{D_j}hX_{\phi_j}w\,dV
 \notag\\
 &=\int_{D_j}h
   \left(\frac{\ii}{2}\bar z(\phi_j)_t-(\phi_j)_z\right)dV
 \notag\\
 &=\int_{D_j}\phi_j
   \left(h_z-\frac{\ii}{2}\bar z h_t\right)dV.
 \label{eq:z-moment}
\end{align}
After division by $M_j$, the first term tends to $h_z(0)$.  The second tends to
zero because $h_t$ is bounded and
$\sup_{D_j}|z|\to0$.  Hence
\begin{equation}\label{eq:hz-zero}
 h_z(0)=0.
\end{equation}
Finally, since $\phi_t(0)=\phi_{z}(0)=0$ and $z(0)=0$, we have $L_\phi|_0= L_0|_0=\partial_{\bar z}$.  Thus $L_{\phi}h|_0 =0$ gives $h_{\bar z}(0)=0$.  Together with \cref{eq:ht-zero} and
\cref{eq:hz-zero}, this proves \cref{eq:dh-zero}.
\end{proof}

\begin{proof}[Proof of \cref{thm:main}]
By \cref{prop:closed-section} and \cref{lem:smallness} we may choose $j_0$ large enough such that the support of the smooth function $\phi$ given in \cref{eq:bumps} is contained in the
prescribed neighborhood $U$, and then choose $\delta$ small enough to guarantee the required
$C^2$ bound and strong pseudoconvexity.  \cref{prop:closed-section}
then gives the required nowhere-zero closed canonical section $\Omega_{\phi}$, and \cref{prop:derivative-killing} shows that every $C^1$ CR function has zero
differential at the origin. It follows immediately that $T_{\phi}^{0,1}$ is not locally CR embeddable at the origin since any local CR embedding into $\mathbb{C}^2$ would have to have zero differential at the origin (its two components being $C^1$ CR functions) and would therefore fail to be an embedding. This proves \cref{thm:main}.
\end{proof}

\section{Globalizing to $S^3$}\label{sec:globalization}

\begin{proof}[Proof of \cref{thm:globalization}]
Regard $S^3$ as the unit sphere in $\C^2$, with complex coordinates
$(z_1,z_2)$, and let $\iota \colon S^3 \to \C^2$ denote the inclusion. The complex $2$-form $\Omega_S=\iota^*(dz_1\wedge dz_2)$ is a global nowhere-zero closed section of the canonical bundle for the standard spherical CR structure on $S^3$. Choose antipodal points $q,p\in S^3$ and a Cayley
transform
\begin{equation*}
 \Psi:\C\times\R\longrightarrow S^3\setminus\{q\}
\end{equation*}
that carries the origin to $p$ and identifies the standard spherical CR
structure with the Heisenberg CR structure. Recall that
\begin{equation*}
 \Omega_0:=\iota_{L_0}\nu=-dz\wedge dw.
\end{equation*}
Since $\Psi^*\Omega_S$ and $\Omega_0$ are nowhere-zero sections of the same
complex line bundle, there is a smooth nowhere-zero complex-valued function
$g$ on $\C\times\R$ such that
\begin{equation}\label{eq:spherical-multiplier}
 \Psi^*\Omega_S=g\Omega_0.
\end{equation}
Both sides of \cref{eq:spherical-multiplier} are closed. Hence
\begin{equation*}
 0=d(g\Omega_0)=dg\wedge\Omega_0=(L_0g)\nu,
\end{equation*}
so that
\begin{equation}\label{eq:g-CR}
 L_0g=0.
\end{equation}

Choose the function $\phi$ in \cref{eq:bumps} with support in a relatively
compact neighborhood of the origin in the Cayley chart. The compactly supported one-form
$\phi\,d\bar z$ defines a compactly supported
one-form $\beta$ on $S^3\setminus\{q\}$ by $\Psi^*\beta = \phi\,d\bar z$. The one-form $\beta$ extends by zero to a one-form on $S^3$, which we also denote by $\beta$. Set
\begin{equation}\label{eq:global-closed-form}
\widehat\Omega=\Omega_S-d\beta.
\end{equation}
The form $\widehat\Omega$ is globally defined and closed. In the Cayley chart,
\cref{prop:closed-section,eq:spherical-multiplier} give
\begin{equation}\label{eq:global-generator}
 \Psi^*\widehat\Omega
 =g\Omega_0-d(\phi\,d\bar z)
 =\iota_{gL_0+X_\phi}\nu.
\end{equation}
For $\|\phi\|_{C^2}$ sufficiently small, $\widehat\Omega$ is nowhere zero and
its kernel defines a strongly pseudoconvex CR structure
$\widehat T^{0,1}$ on $S^3$. Indeed, in the Cayley chart this line field is
spanned by
\begin{equation*}
 gL_0+X_\phi
 =g\bigl(L_0+g^{-1}X_\phi\bigr),
\end{equation*}
and since $g^{-1}$ and its first derivatives are bounded on the fixed compact support of $\phi$, the corresponding CR structure is therefore a $C^1$-small perturbation of the Heisenberg CR structure when $\|\phi\|_{C^2}$ is small. Outside
the support of $\phi$ it is the standard spherical CR structure. Thus, by
\cref{lem:smallness}, the perturbation can be taken to be arbitrarily small in
the $C^1$ norm.

It remains to prove nonembeddability at $p$. Let $h$ be a $C^1$ CR function
defined near $p$, and denote its pullback to the Cayley chart again by $h$.
Then
\begin{equation}\label{eq:global-CR-equation}
 (gL_0+X_\phi)h=0.
\end{equation}
For all sufficiently large $j$, the solid torus $D_j$ lies in the domain of
$h$. On the set $w^{-1}(G)$ used in the proof of
\cref{cor:weighted-bulk}, one has $X_\phi=0$, and therefore
$L_0h=0$ there. In view of \cref{eq:g-CR,eq:flat-first-integrals}, the two
functions
\begin{equation*}
 gzh,\qquad gwh
\end{equation*}
are also annihilated by $L_0$ on $w^{-1}(G)$. Applying
\cref{lem:fiber-period} to these two functions gives
\begin{equation}\label{eq:global-weighted-bulk}
 0=\int_{D_j}L_0(gzh)\,dV
   =\int_{D_j}gzL_0h\,dV,
 \qquad
 0=\int_{D_j}L_0(gwh)\,dV
   =\int_{D_j}gwL_0h\,dV.
\end{equation}
On $D_j$ only the summand $\delta\phi_j$ is nonzero, so
\cref{eq:global-CR-equation} reads
\begin{equation*}
 gL_0h=-\delta X_{\phi_j}h.
\end{equation*}
Consequently, \cref{eq:global-weighted-bulk} implies
\begin{equation*}
 \int_{D_j}zX_{\phi_j}h\,dV=0
 \qquad \text{and} \qquad 
 \int_{D_j}wX_{\phi_j}h\,dV=0.
\end{equation*}
The integration-by-parts arguments in the proof of
\cref{prop:derivative-killing} now apply without change and yield
\begin{equation*}
 \int_{D_j}\phi_jh_t\,dV=0
 \qquad \text{and} \qquad 
 \int_{D_j}\phi_j
 \left(h_z-\frac{\ii}{2}\bar z h_t\right)dV=0.
\end{equation*}
Dividing by the masses $M_j$ in \cref{eq:masses} and letting $j\to\infty$
gives
\begin{equation*}
 h_t(0)=h_z(0)=0.
\end{equation*}
Finally, $X_\phi$ vanishes at the origin and $g(0)\ne0$, so
\cref{eq:global-CR-equation} also gives $h_{\bar z}(0)=0$. Thus every $C^1$ CR
function near $p$ has zero differential at $p$. It follows that
$\widehat T^{0,1}$ is not locally embeddable at $p$, completing the proof.
\end{proof}


\begin{thebibliography}{99}

\bibitem{AlekseevskyGanjiSchmalz2018}
D.~V.~Alekseevsky, M.~Ganji, and G.~Schmalz,
\emph{CR-geometry and shearfree Lorentzian geometry},
in \emph{Geometric Complex Analysis}, Springer Proc. Math. Stat. \textbf{246},
Springer, Singapore, 2018, 11--22;
\href{https://doi.org/10.1007/978-981-13-1672-2_2}
{doi:10.1007/978-981-13-1672-2\_2}.

\bibitem{BiquardHerzlich2005}
O.~Biquard and M.~Herzlich,
\emph{A Burns--Epstein invariant for ACHE 4-manifolds},
Duke Math.\ J.\ \textbf{126} (2005), 53--100;
\href{https://doi.org/10.1215/S0012-7094-04-12612-0}
{doi:10.1215/S0012-7094-04-12612-0};
\href{https://arxiv.org/abs/math/0111218}{arXiv:math/0111218}.

\bibitem{Bryant2004}
R.~L.~Bryant,
\emph{Real hypersurfaces in unimodular complex surfaces},
preprint (2004);
\href{https://arxiv.org/abs/math/0407472}{arXiv:math/0407472}.

\bibitem{CaseYang2013}
J.~S.~Case and P.~C.~Yang,
\emph{A Paneitz-Type Operator for CR Pluriharmonic Functions},
Bull.\ Inst.\ Math.\ Acad.\ Sin.\ (N.S.) \textbf{8} (2013), 285--322;
\href{https://arxiv.org/abs/1309.2528}{arXiv:1309.2528}.

\bibitem{ChengYau1980}
S.-Y.~Cheng and S.-T.~Yau,
\emph{On the existence of a complete K\"ahler metric on non-compact
complex manifolds and the regularity of Fefferman's equation},
Comm. Pure Appl. Math. \textbf{33} (1980), 507--544;
\href{https://doi.org/10.1002/cpa.3160330404}
{doi:10.1002/cpa.3160330404}.

\bibitem{Cho1998}
S.~Cho,
\emph{Extension of CR structures on three dimensional pseudoconvex CR
manifolds},
Nagoya Math. J. \textbf{152} (1998), 97--129.
\href{https://doi.org/10.1017/S0027763000006814}
{doi:10.1017/S0027763000006814}.

\bibitem{Farris1986}
F.~A.~Farris,
\emph{An intrinsic construction of Fefferman's CR metric},
Pacific J. Math. \textbf{123} (1986), 33--45;
\href{https://doi.org/10.2140/pjm.1986.123.33}
{doi:10.2140/pjm.1986.123.33}.

\bibitem{Fefferman1976}
C.~L.~Fefferman,
\emph{Monge--Amp\`ere equations, the Bergman kernel, and geometry of
pseudoconvex domains},
Ann. of Math. (2) \textbf{103} (1976), 395--416;
\href{https://doi.org/10.2307/1970945}
{doi:10.2307/1970945};
correction, Ann. of Math. (2) \textbf{104} (1976), 393--394;
\href{https://doi.org/10.2307/1970961}
{doi:10.2307/1970961}.

\bibitem{GoldbergSachs1962}
J.~N.~Goldberg and R.~K.~Sachs,
\emph{A theorem on Petrov types},
Acta Phys.\ Polon.\ (Suppl.) \textbf{22}  (1962), 13--23.

\bibitem{GoldbergSachs2009}
J.~N.~Goldberg and R.~K.~Sachs,
\emph{Republication of: A theorem on Petrov types},
Gen. Relativity Gravitation \textbf{41} (2009), 433--444;
\href{https://doi.org/10.1007/s10714-008-0722-5}
{doi:10.1007/s10714-008-0722-5}.

\bibitem{HillLewandowskiNurowski2008}
C.~D.~Hill, J.~Lewandowski, and P.~Nurowski,
\emph{Einstein's equations and the embedding of 3-dimensional CR manifolds},
Indiana Univ.\ Math.\ J.\ \textbf{57} (2008), 3131--3176;
\href{https://doi.org/10.1512/iumj.2008.57.3473}
{doi:10.1512/iumj.2008.57.3473};
\href{https://arxiv.org/abs/0709.3660}{arXiv:0709.3660}.

\bibitem{HillNacinovich2013}
C.~D.~Hill and M.~Nacinovich,
\emph{Non completely solvable systems of complex first order PDEs},
Rend. Semin. Mat. Univ. Padova \textbf{129} (2013), 129--170.
\href{https://doi.org/10.4171/RSMUP/129-9}
{doi:10.4171/RSMUP/129-9};
\href{https://arxiv.org/abs/1111.1502}{arXiv:1111.1502}

\bibitem{Hirachi2013}
K.~Hirachi,
\emph{$Q$-prime curvature on CR manifolds},
Differential Geom. Appl. \textbf{33} (2014), suppl., 213--245;
\href{https://doi.org/10.1016/j.difgeo.2013.10.013}
{doi:10.1016/j.difgeo.2013.10.013};
\href{https://arxiv.org/abs/1302.0489}{arXiv:1302.0489}.

\bibitem{Jacobowitz1987}
H.~Jacobowitz,
\emph{The canonical bundle and realizable CR hypersurfaces},
Pacific J.\ Math.\ \textbf{127} (1987), 91--101;
\href{https://doi.org/10.2140/pjm.1987.127.91}
{doi:10.2140/pjm.1987.127.91}.

\bibitem{Jacobowitz2020}
H.~Jacobowitz,
\emph{A conjecture of Trautman},
Lect.\ Notes Semin.\ Interdiscip.\ Mat.\ \textbf{15} (2020), 33--43;
\href{https://arxiv.org/abs/1902.00559}{arXiv:1902.00559}.

\bibitem{JacobowitzTreves1982}
H.~Jacobowitz and F.~Tr\`eves,
\emph{Non-realizable CR structures},
Invent.\ Math.\ \textbf{66} (1982), 231--249;
\href{https://doi.org/10.1007/BF01389393}{doi:10.1007/BF01389393}.

\bibitem{Kiremidjian1979}
G.~K.~Kiremidjian,
\emph{A direct extension method for CR structures},
Math. Ann. \textbf{242} (1979), 1--19;
\href{https://doi.org/10.1007/BF01420478}
{doi:10.1007/BF01420478}.

\bibitem{Lee1988}
J.~M.~Lee,
\emph{Pseudo-Einstein structures on CR manifolds},
Amer.\ J.\ Math.\ \textbf{110} (1988), 157--178;
\href{https://doi.org/10.2307/2374543}{doi:10.2307/2374543}.

\bibitem{LewandowskiNurowskiTafel1990}
J.~Lewandowski, P.~Nurowski, and J.~Tafel,
\emph{Einstein's equations and realizability of CR manifolds},
Class.\ Quantum Grav.\ \textbf{7} (1990), L241--L246.

\bibitem{LewandowskiNurowskiTafel1991}
J.~Lewandowski, P.~Nurowski, and J.~Tafel,
\emph{Algebraically special solutions of the Einstein equations with pure radiation fields},
Class.\ Quantum Grav.\ \textbf{8} (1991), 493--501;
\href{https://doi.org/10.1088/0264-9381/8/3/009}
{doi:10.1088/0264-9381/8/3/009}.

\bibitem{Lewy1957}
H.~Lewy,
\emph{An example of a smooth linear partial differential equation without
solution},
Ann.\ of Math.\ (2) \textbf{66} (1957), 155--158;
\href{https://doi.org/10.2307/1970121}{doi:10.2307/1970121}.

\bibitem{Nirenberg1974}
L.~Nirenberg,
\emph{On a question of Hans Lewy},
Russian Math.\ Surveys \textbf{29} (1974), 251--262.

\bibitem{Nirenberg1975}
L.~Nirenberg,
\emph{On a problem of Hans Lewy},
in \emph{Fourier Integral Operators and Partial Differential Equations},
Lecture Notes in Math.\ \textbf{459}, Springer, Berlin, 1975, 224--234;
\href{https://doi.org/10.1007/BFb0074196}{doi:10.1007/BFb0074196}.

\bibitem{Nurowski1996}
P.~Nurowski,
\emph{Optical geometries and related structures},
J.\ Geom.\ Phys.\ \textbf{18} (1996), 335--348;
\href{https://doi.org/10.1016/0393-0440(95)00012-7}
{doi:10.1016/0393-0440(95)00012-7}.

\bibitem{Robinson1959}
I.~Robinson,
\emph{A solution of the Maxwell--Einstein equations},
Bull.\ Acad.\ Polon.\ Sci.\ S\'er.\ Sci.\ Math.\ Astronom.\ Phys.\ \textbf{7} (1959),
351--352.

\bibitem{Robinson1961}
I.~Robinson,
\emph{Null electromagnetic fields},
J.\ Math.\ Phys.\ \textbf{2} (1961), 290--291;
\href{https://doi.org/10.1063/1.1703712}{doi:10.1063/1.1703712}.

\bibitem{RobinsonTrautman1960}
I.~Robinson and A.~Trautman,
\emph{Spherical gravitational waves},
Phys.\ Rev.\ Lett.\ \textbf{4} (1960), 431--432;
\href{https://doi.org/10.1103/PhysRevLett.4.431}{doi:10.1103/PhysRevLett.4.431}.

\bibitem{RobinsonTrautman1962}
I.~Robinson and A.~Trautman,
\emph{Some spherical gravitational waves in general relativity},
Proc.\ Roy.\ Soc.\ London Ser.\ A \textbf{265} (1962), 463--473;
\href{https://doi.org/10.1098/rspa.1962.0036}{doi:10.1098/rspa.1962.0036}.

\bibitem{RobinsonTrautman1986}
I.~Robinson and A.~Trautman,
\emph{Cauchy--Riemann structures in optical geometry},
in \emph{Proceedings of the Fourth Marcel Grossmann Meeting on General
Relativity}, 1986, 317--324.

\bibitem{Rosay1989}
J.-P.~Rosay,
\emph{New examples of non-locally embeddable CR structures
(with no non-constant CR distributions)},
Ann.\ Inst.\ Fourier (Grenoble) \textbf{39} (1989), 811--823;
\href{https://doi.org/10.5802/aif.1189}{doi:10.5802/aif.1189}.

\bibitem{SchmalzGanji2019}
 G.~Schmalz and M.~Ganji,
\emph{A criterion for local embeddability of three-dimensional CR structures},
Ann.\ Mat.\ Pura Appl.\ (4) \textbf{198} (2019), 491--503;
\href{https://doi.org/10.1007/s10231-018-0785-1}
{doi:10.1007/s10231-018-0785-1};
\href{https://arxiv.org/abs/1803.01535}{arXiv:1803.01535}.

\bibitem{Tafel1985}
J.~Tafel,
\emph{On the Robinson theorem and shearfree geodesic null congruences},
Lett.\ Math.\ Phys.\ \textbf{10} (1985), 33--39;
\href{https://doi.org/10.1007/BF00704584}{doi:10.1007/BF00704584}.

\bibitem{TaghaviChabert2026}
A.~Taghavi-Chabert,
\emph{Perturbations of Fefferman spaces over CR three-manifolds},
Trans.\ Amer.\ Math.\ Soc.\ \textbf{379} (2026), no.~2, 923--978;
\href{https://doi.org/10.1090/tran/9498}
{doi:10.1090/tran/9498};
\href{https://arxiv.org/abs/2303.07328}{arXiv:2303.07328}.

\bibitem{Trautman1999}
A.~Trautman,
\emph{On complex structures in physics},
in \emph{On Einstein's Path}, Springer, New York, 1999, 487--501;
\href{https://doi.org/10.1007/978-1-4612-1422-9_34}
{doi:10.1007/978-1-4612-1422-9\_34};
\href{https://arxiv.org/abs/math-ph/9809022}{arXiv:math-ph/9809022}.

\bibitem{Trautman2002}
A.~Trautman,
\emph{Robinson manifolds and Cauchy--Riemann spaces},
Class.\ Quantum Grav.\ \textbf{19} (2002), R1--R10;
\href{https://doi.org/10.1088/0264-9381/19/2/201}
{doi:10.1088/0264-9381/19/2/201}.

\end{thebibliography}
\end{document}